\documentclass{elsarticle}
\usepackage{amsthm, amssymb,verbatim}
\usepackage{enumitem}
\setlist[enumerate]{label=(\arabic*)}

\newtheorem{thm}{Theorem}[section]
\newtheorem{prop}[thm]{Proposition}
\newtheorem{cor}[thm]{Corollary}
\newtheorem{lem}[thm]{Lemma}
\newtheorem{rem}[thm]{Remark}

\providecommand{\acf}{\Bbbk}
\providecommand{\HH}{\mathbb{H}}
\providecommand{\CC}{\mathbb{C}}
\providecommand{\RR}{\mathbb{R}}

\providecommand{\NN}{\mathbb{N}}
\providecommand{\hr}{\HH_R}
\providecommand{\ccr}{\CC_R}
\providecommand{\ii}{\mathbf{i}}
\providecommand{\jj}{\mathbf{j}}
\providecommand{\kk}{\mathbf{k}}
\providecommand{\id}{\mathrm{id}}

\providecommand{\eqref}[1]{(\ref{#1})}

\begin{document}
%\title[Quaternionic and matrix nullstellens\"atze]{Parallels between quaternionic and matrix nullstellens\"atze}
\title{Parallels between quaternionic and matrix nullstellens\"atze}
\author{Jakob Cimprič\fnref{fn1,fn2} }
%\thanks{The author is affiliated with \textit{University of Ljubljana, Faculty of Mathematics and Physics, Slovenia}
%and \textit{Institute of Mathematics, Physics and Mechanics, Slovenia}}
%\affiliation{organization={University of Ljubljana, Faculty of Mathematics and Physics, and Institute of %Mathematics, Physics and Mechanics}, addressline={Jadranska 19}, city={Ljubljana}, country={Slovenia}}
\fntext[fn1]{The author is affiliated with \textit{University of Ljubljana, Faculty of Mathematics and Physics} 
and \textit{Institute of Mathematics, Physics and Mechanics}; both in Ljubljana, Slovenia.}
\fntext[fn2]{The author acknowledges the financial support from the \textit{Slovenian Research and Innovation Agency}  (research core funding no. P1-0222 and grants no. J1-50002 and no. J1-60011)}
%\thanks{16D25, 16D40, 16S50, 14A22}
\ead{ jaka.cimpric@fmf.uni-lj.si}

%\address{University of Ljubljana, Faculty of Mathematics and Physics,
%Jadranska ulica 19, SI-1000 Ljubljana, Slovenia}
%\email{jaka.cimpric@fmf.uni-lj.si}
%\address{ Institute of Mathematics, Physics and Mechanics,
%Jadranska ulica 19, SI-1000 Ljubljana, Slovenia}

\begin{abstract}
We prove a new quaternionic and a new matrix nullstellensatz.
We also show that both theories are intertwined.
For every $g_1,\ldots,g_m,f \in \HH[x_1,\ldots,x_d]$ (where $x_1,\ldots,x_d$ are central) we show that the following are equivalent:
(a) For every $a \in \HH^d$ whose components pairwise commute and which satisfies $g_1(a)=\ldots=g_m(a)=0$
we have  $f(a)=0$. (b) $f$ belongs to the smallest semiprime left  ideal containing $g_1,\ldots,g_m$.
On the other hand, for every $G_1,\ldots,G_m, F \in M_n(\acf[x_1,\ldots,x_d])$, where $\acf$ is an algebraically closed field,
we show that the following are equivalent (where $I$ is the left ideal generated by $G_1,\ldots,G_m$):
(a) For every $a \in \acf^d$ and $v \in \acf^n$ such that $G_1(a)v=\ldots=G_m(a)v=0$, we have $F(a)v=0$.
(b) For every $A \in M_n(\acf)$ there exists $N \in \NN_0$ such that $(AF)^N \in I+I(AF)+\ldots+I(AF)^N$.
\end{abstract}

\maketitle

\section{Introduction}
The aim of this paper is to discuss and improve on three recent generalizations of Hilbert's Nullstellensatz.
The first and the second one extend it from complex polynomials to central quaternionic polynomials
(see the Strong Nullstellensatz \cite[Theorem 1.2]{ap1} by G. Alon and E. Paran  and 
the Explicit Hilbert’s Nullstellensatz over quaternions \cite[Theorem 1.5]{aya} by M. Aryapoor; i.e. claims (1) and (2) of Theorem \ref{thmb} below)
while the third one extends it to matrix polynomials 
(see the one-sided Hilbert’s Nullstellensatz for matrix polynomials \cite[Theorem 4]{cimpric} by the author; i.e. claims (3) and (4) of Theorem \ref{thmd} below.)

It is interesting to note that these  results went into completely different directions.
The results about quaternionic polynomials are about  an explicit version and about completely prime left ideals while the results about matrix polynomials are about prime and semiprime left ideals.
For quaternionic polynomials we will complement these results with results about  prime and semiprime left ideals
(see claims (3) and (4) of Theorem \ref{thmb})  while for matrix polynomials we will complement them with an explicit version and results about completely prime ideals (see claims (1) and (2) of Theorem \ref{thmd}).
In summary, we will show that the results about quaternionic and matrix polynomials
are  entirely parallel which suggests a possible existence of a deeper theory.

We start with some notation.  
Let $\HH[x_1,\ldots,x_d]$ be the ring of quaternionic polynomials in $d$ central variables.
For every quaternionic polynomial
\begin{equation}
\label{def1}
f =\sum_{i_1,\ldots,i_d} c_{i_1,\ldots,i_d} x_1^{i_1} \cdots x_d^{i_d} \in \HH[x_1,\ldots,x_d]
\end{equation}
and every point $a=(a_1,\ldots,a_d) \in \HH^d$, we define the value $f(a)$ by
\begin{equation}
\label{def2}
f(a)=\sum_{i_1,\ldots,i_d} c_{i_1,\ldots,i_d} a_1^{i_1} \cdots a_d^{i_d} \in \HH.
\end{equation}
A point $a=(a_1,\ldots,a_d) \in \HH^d$ is \textit{central}  iff $a_i a_j=a_j a_i$ for all $i$ and $j$.
(This terminology is from \cite{ap3}. In \cite{aya}, such a point is ``good''.)
The set of all central points will be denoted by $\HH_c^d$. Pick a central point $a$
and quaternionic polynomials $f=\sum_\alpha c_\alpha x^\alpha$  and $g=\sum_\beta d_\beta x^\beta$ (in multiindex notation) and note that 
\begin{equation}
\label{def4}
(fg)(a) =\sum_\alpha \sum_\beta c_\alpha d_\beta a^{\alpha+\beta}=\sum_\alpha \sum_\beta c_\alpha d_\beta a^\beta a^\alpha = \sum_\alpha c_\alpha g(a) a^\alpha
\end{equation}
which implies that
\begin{comment}
\begin{equation}
\label{eval}
(fg)(a)=\left\{ \begin{array}{cc} 0 & \mbox{ if } g(a)=0 \\ f(b)g(a) & \mbox{ if } g(a) \ne 0 \end{array} \right.
\end{equation}
where $b=g(a) a g(a)^{-1}$.
%$b=(g(a) a_1 g(a)^{-1},\ldots,g(a) a_d g(a)^{-1})$.  
It follows that 
\end{comment}
the set
\begin{equation}
\label{def3}
I_a:=\{ g \in \HH[x_1,\ldots,x_d] \mid g(a)=0\}
\end{equation}
is a left ideal for every central point $a$.
%If $a$ is not central then $1 \in I_a$.

Theorem \ref{thma} is the Weak Nullstellensatz from \cite[Theorem 1.1]{ap1}.

\begin{thm}
\label{thma}
For every $a \in \HH_c^d$ the left ideal $I_a$ is maximal.  Moreover, every maximal left ideal of $\HH[x_1,\ldots,x_d]$
is of this form. In particular, for every proper left ideal of $\HH[x_1,\ldots,x_d]$
there exists a central point annihilating all of its elements.
\end{thm}

For every left ideal $I$ of $\HH[x_1,\ldots,x_d]$ we write
$\mathcal{Z}(I)$ for the set of all central points that are annihilated by all elements of $I$.
For every subset $Z$ of $\HH^d_c$ we write $\mathcal{I}(Z)$ for the set of
all quaternionic polynomials that annihilate all central points from $Z$.
The left ideal
\begin{equation}
\sqrt{I} := \mathcal{I}(\mathcal{Z}(I))
\end{equation}
will be called \textit{the radical} of a given left ideal $I$.
Our aim is to give several characterizations of the radical.
We start with  two trivial ones, then we recall two recent ones
and finally we give two new characterizations.

As $\mathcal{Z}(I)=\{a \in \HH^d_c \mid I \subseteq I_a\}$ for every left ideal $I$ 
and $\mathcal{I}(Z)=\bigcap_{a \in Z} I_a$ for every set of central points $Z$,
we have $\sqrt{I} = \bigcap_{I \subseteq I_a} I_a$. 
 Theorem \ref{thma} now implies that $\sqrt{I}$
is the intersection of all maximal left ideals that contain $I$.

%By \cite[Theorem 2.9 in $\S$1]{mr}, 
Recall that the ring $\HH[x_1,\ldots,x_d]$ is left Noetherian; see \cite[$\S$1.2.9]{mr}.
 If a left ideal $I$ of $\HH[x_1,\ldots,x_d]$ is generated by  $g_1,\ldots,g_m$
then $\sqrt{I}$ consists of all $f$ such that for every $a \in \HH^d_c$ satisfying
$g_1(a)=\ldots=g_m(a)=0$ we have $f(a)=0$.

In 2021, G. Alon and E. Paran proved that $\sqrt{I}$ is equal to the intersection
of all completely prime left ideals that contain $I$; see  \cite[Theorem 1.2]{ap1}.
We recall their result in claim (2) of Theorem \ref{thmb}.
They used the definition of a completely prime left ideal by M. L. Reyes \cite{reyes};
we will recall it  in Section \ref{sec3}.

In 2024, M. Aryapoor gave an explicit characterization of $\sqrt{I}$; see \cite[Theorem 1.5]{aya}.
We will recall his result in claim (1) of Theorem \ref{thmb}.

We will give two new characterizations of $\sqrt{I}$. By Theorem \ref{thm2}, 
$\sqrt{I}$ is equal to the intersection of all prime left ideals that contain $I$.
By Theorems \ref{thm1} and \ref{thm2}, $\sqrt{I}$ is the smallest semiprime left ideal containing $I$.
These  characterizations also appear as claims (3) and (4) of Theorem \ref{thmb}. 
We use the definitions of a prime and a semiprime left ideal by F. Hansen \cite{hansen};
we recall them in Section \ref{sec1}.

\begin{thm}
\label{thmb}
For every left ideal $I$ of $\HH[x_1,\ldots,x_d]$,  $\sqrt{I}$ is equal to
%Pick any $g_1,\ldots,g_m \in \HH[x_1,\ldots,x_d]$ and write $I$ for the left ideal generated by them.
%For every $f \in \HH[x_1,\ldots,x_d]$ the following are equivalent.
\begin{enumerate}
\item the set of all $f \in \HH[x_1,\ldots,x_d]$ such that for every $b \in \HH$ there exists $N \in \NN_0$ satisfying 
$(bf)^N \in I+I(bf)+\ldots+I(bf)^N.$
\item  the intersection of all completely prime left ideals that contain~$I$.
\item  the intersection of all  prime left ideals that contain $I$.
\item the smallest semiprime left ideal that contains $I$.
\end{enumerate}
\end{thm}

We will also discuss the analogues of Theorems \ref{thma} and \ref{thmb} for matrix polynomials.
Let $\acf$ be an algebraically closed field, $\acf[x_1,\ldots,x_d]$ the usual ring of polynomials 
and  $M_n(\acf[x_1,\ldots,x_d])$ the ring of  $n \times n$ matrix polynomials.
A \textit{directional point} is a pair $(a,v)$ where $a \in \acf^d$ is a usual point and $v \in \acf^n$
is a direction (i.e. a nonzero vector). The value of a $n \times n$ matrix polynomial $F$ in 
a directional point $(a,v)$ is defined as the vector $F(a)v \in \acf^n$.
For every directional point $(a,v)$, the set
\begin{equation}
\label{def5}
D(a,v):=\{F \in M_n(\acf[x_1,\ldots,x_d]) \mid F(a)v=0\}
\end{equation}
of all matrix polynomials whose value at $(a,v)$ is zero is a left ideal.

Theorem \ref{thmc} is a weak nullstellensatz for matrix polynomials.
It is a special case of \cite[Theorem 1.2]{stone}. For an
alternative proof see \cite[Proposition 2 and Remark 1]{cimpric}. 

\begin{thm}
\label{thmc}
For every directional point $(a,v)$, the left ideal $D(a,v)$ is maximal. Moreover, every maximal left
ideal of  $M_n(\acf[x_1,\ldots,x_d])$ is of this form. In particular, for every proper left ideal $I$
of  $M_n(\acf[x_1,\ldots,x_d])$ there exists a directional point $(a,v)$ such that $F(a)v=0$
for all $F \in I$.
\end{thm}

For every left ideal $I$ of $M_n(\acf[x_1,\ldots,x_d])$ we define its \textit{radical} $\sqrt{I}$ 
as the set of all $F \in M_n(\acf[x_1,\ldots,x_d])$ which annihilate all directional points 
that are annihilated by all elements of $I$. 

By definition, $\sqrt{I}$ is equal to the intersection of all left ideals of the form $D(a,v)$ which contain $I$.
By Theorem \ref{thmc} it follows that $\sqrt{I}$ is the intersection of  all maximal left ideals that contain~$I$.

As $M_n(\acf[x_1,\ldots,x_d])$ is left Noetherian by \cite[$\S$1.1.2]{mr}, 
every left ideal $I$ of $M_n(\acf[x_1,\ldots,x_d])$ is generated by some $G_1,\ldots,G_m$. 
An element  $F$ belongs to $\sqrt{I}$ iff  for every $(a,v)$ satisfying $G_1(a)v=\ldots=G_m(a)v=0$ we have $F(a)v=0$.
 
In 2022 we proved that for every left ideal $I$ of $M_n(\acf[x_1,\ldots,x_d])$, $\sqrt{I}$ 
is equal to the intersection of all prime left ideals that contain $I$ and it is 
also equal to the smallest semiprime left ideal containing $I$; see \cite[Theorem 3]{cimpric}.
We recall these results as claims (3) and (4) of Theorem \ref{thmd}.

We will also give two new characterizations of $\sqrt{I}$ for matrix polynomials.
By Corollary \ref{cor3}, $\sqrt{I}$ is the intersection of all completely prime left ideals containing $I$.
Theorem \ref{thm4} gives an explicit characterization of $\sqrt{I}$ which is similar
to the result \cite[Theorem 1.5]{aya} of M. Aryapoor for quaternionic polynomials.
These characterizations also appear as claims (2) and (1) of Theorem \ref{thmd}.

\begin{thm}
\label{thmd}
For every left ideal $I$ of $M_n(\acf[x_1,\ldots,x_d])$, $\sqrt{I}$ is~equal~to
\begin{enumerate}
\item the set of all $F \in M_n(\acf[x_1,\ldots,x_d])$ such that for every ${A \in M_n(\acf)}$ there is $N \in \NN_0$ satisfying
$(AF)^N \in I+I(AF)+\ldots+I(AF)^N.$
\item the intersection of all completely prime left ideals that contain~$I$.
\item the intersection of all  prime left ideals that contain $I$.
\item the smallest semiprime left ideal that contains $I$.
\end{enumerate}
\end{thm}

Another description of the radical of a left ideal $I$ in $\HH[x_1,\ldots,x_d]$ is discussed in \cite{slice}, \cite{ap2}, \cite{ap3}.  They show that $\sqrt{I}$ consists of all quaternionic polynomials that annihilate all points from $\HH^d$ (not just $\HH^d_c$!) that are annihilated by all elements of $I$. We do not know if this result has a parallel for matrix polynomials.

\section{Prime and semiprime left ideals and submodules}
\label{sec1}

Let $A$ be an associative unital ring and let $I$ be a left ideal of $A$.
We say that $I$ is \textit{prime} iff for every $a,b \in A$ such that $aAb \subseteq I$ we have $a  \in A$ or $b \in A$.
We say that $I$ is \textit{semiprime} if for every $a \in A$ such that $aAa \subseteq I$ we have  $a \in I$.
Both definitions are from \cite{hansen}. For two-sided ideals they coincide with the usual definitions.

Let $R$ be a commutative ring and let $\HH_R$ be the standard quaternion $R$-algebra, i.e.
$$\hr=\{r_0+r_1 \ii +r_2 \jj +r_3 \kk  \mid r_0,r_1,r_2,r_3 \in R\}$$
where $\ii ^2=\jj ^2=\kk ^2=\ii  \jj  \kk =-1$ and $\ii ,\jj ,\kk $ commute with $R$. 
For every $a=a_0+a_1 \ii+a_2 \jj+a_3 \kk \in \hr$ we write $\bar{a}=a_0-a_1 \ii-a_2 \jj-a_3 \kk$.
Note that $\hr=R^4$ as an $R$-module. Note also that
$\HH_{\RR[x_1,\ldots,x_d]}=\HH[x_1,\ldots,x_d]$.

The aim of this section is to prove the following result.
%which shows that claim (4) of Theorem \ref{thmb} follows from claim (3).

\begin{thm} 
\label{thm1}
Suppose that $\frac12 \in R$. A left  ideal of $\hr$ is semiprime iff it is an intersection of prime left ideals.
\end{thm}

\begin{rem}
\label{rem00}
We need $\frac12 \in R$ because we use the well-known identities
\begin{eqnarray*}
a_0=\frac14(a-\ii a \ii-\jj a \jj-\kk a \kk)\\
a_1=\frac14(\jj a \kk - a \ii-\ii a-\kk a \jj) \\
a_2=\frac14(\kk a \ii-a \jj-\jj a-\ii a \kk) \\
a_3=\frac14(\ii a \jj-a \kk-\kk a-\jj a \ii)
\end{eqnarray*}
%from \cite{ap1} 
where  $a=a_0+a_1 \ii+a_2 \jj+a_3 \kk$. 
\end{rem}

We will split the proof into several lemmas. %We assume throughout that $\frac12 \in R$.
We start with useful characterizations of prime and semiprime left ideals in $\HH_R$.
Recall the following  definitions of  prime and semiprime submodules of $R^n$ from \cite{cimpric}.
We say that a submodule $N$  of $R^n$ is \textit{prime} if for every $r\in R$ and $a \in R^n$
such that $ra \in N$ we have either $rR^n \subseteq N$ or $a \in N$. We say that a submodule 
$N$ of $R^n$ is \textit{semiprime} if for every $a=(a_1,\ldots,a_n) \in R^n$ such that $a_i a\in N$
for all $i$ we have  $a \in N$.

\begin{lem}
\label{lem1}
For every left ideal $I$ of $\hr$ the following are equivalent.
\begin{enumerate}
\item $I$ is  prime as a left ideal.
\item For every $r \in R$ and $a \in \hr$ such that $ra \in I$ we have  $r \in I$ or $a \in I$
(i.e. $I$ is prime as a submodule of $\hr$.)
\end{enumerate}
\end{lem}

\begin{proof}
Suppose that (1) is true and pick any $r \in R$ and $a \in \hr$ such that $ra \in I$. It follows that $r \hr a \subseteq I$ which implies that either $r \in I$ or $a \in I$. So (2) is true.

Conversely, suppose that (2) is true. Pick any $a,b \in \hr$ such that $a \hr b \subseteq I$. It follows
that $(\sum \hr a \hr) b \subseteq I$. By Remark \ref{rem00} we have $a_i b \in I$ for $i=0,1,2,3$. 
If $b \not\in I$ then, by (2), $a_i \in I$ for all $i$. Since $I$ is a left ideal, it follows that $a \in I$.
Thus, (1) is true.
\end{proof}

A similar characterization also exists  for semiprime ideals.

\begin{lem}
\label{lem2}
For every left ideal $I$ of $\hr$ the following are equivalent.
\begin{enumerate}
\item $I$ is semiprime as a left ideal.
\item For every  $a=a_0+a_1 \ii +a_2 \jj +a_3 \kk  \in \hr$ such that $a_i a \in I$ for all $i$
we have $a \in I$  (i.e. $I$ is semiprime as a submodule of $\hr$.)
\end{enumerate}
\end{lem}

\begin{proof}
Suppose that (1) is true and pick any  $a \in \hr$ such that $a_i a \in I$ for all $i$.
 It follows that $a_i \hr a \subseteq I$ which implies that $a \hr a \subseteq I$.
By assumption, it follows that $a \in I$. This proves (2).

Suppose now that (2) is true.. Pick any $a \in \hr$ such that $a \hr a \subseteq I$. It follows
that $(\sum \hr a \hr) a \subseteq I$. By Remark \ref{rem00}, we have $a_i a \in I$ for all $i$.
It follows by (2) that $a \in I$. So, (1) is true.
\end{proof}

We will also need the following technical lemma.

\begin{lem}
\label{lem3}
Let $I$ be a left ideal of $\hr$ and let $N$ be a prime submodule of $\hr$ containing $I$.
Then the set  $N \cap \ii  N \cap \jj  N \cap \kk  N$ is a prime left ideal of $\hr$ containing $I$.
\end{lem}

\begin{proof}
Since $N$ is a submodule of $\hr$,  the sets $\ii  N$, $\jj  N$ and $\kk  N$ are also submodules of $\hr$.
Note that the submodule $$J:= N \cap \ii  N \cap \jj  N \cap \kk  N$$
satisfies $\ii J \subseteq J$, $\jj J \subseteq J$ and $\kk J \subseteq J$
which implies that $J$ is a left ideal of $\hr$. Note also that $I=\ii I=\jj I=\kk I$ since $I$ is a left ideals.
Therefore, $I \subseteq N$ implies $I \subseteq J$.

Suppose now that $N$ is a prime submodule of $\hr$. To show that $J$ is prime as a left ideal
it suffices by Lemma \ref{lem1} to show that it is prime as a submodule of $\hr$.
Pick $r \in R$ and $a \in \hr$ such that $r a \in J$ and $a \not\in J$. We consider four cases.
If $a \not\in N$ then $ra \in N$ implies that $r \hr \subseteq N$ by the definition of a prime submodule. 
In particular, $r,r \ii, r \jj,r \kk \in N$ which implies $r \in J$. 
If $a \not\in \ii N$ then $r a \in \ii N$ also implies  that $r \hr \subseteq N$ 
(since  $\ii a \not\in N$ and $r(\ii a) \in N$).
So, $r \in J$ in this case, too.
The remaining cases $a \not\in \jj N$ and $a \not\in \kk N$ are analogous.
\end{proof}

We are now ready for the proof of Theorem \ref{thm1}.

\begin{proof}
Let $I$ be a semiprime left ideal of $\hr$. By Lemma \ref{lem2}, $I$ is a semiprime submodule of $\hr$.
By  \cite[Theorem 1]{cimpric}, $I$ is an intersection of prime submodules of $\hr$.
For every prime submodule $N$ containing $I$ there exists by Lemma \ref{lem3} 
a prime left ideal between $I$ and $N$. It follows that $I$ is an intersection of prime left ideals.
The converse is clear.
\end{proof}

\section{Maximal left ideals}
\label{sec2}

We continue to assume that $R$ is a commutative ring containing $\frac12$.
In this section we also assume that the ring $R$ is Jacobson, i.e. every prime ideal of $R$ is an intersection 
of maximal ideals. We want to show that in this case the ring $\hr$ is  ``left Jacobson\rq\rq, i.e. every prime left ideal
of $\hr$ is an intersection of maximal left ideals.

\begin{thm}
\label{thm2}
If $R$ is a Jacobson ring containing $\frac12$ then every prime left ideal of $\hr$ is equal to an intersection
of maximal left ideals.
\end{thm}

Since the ring $\RR[x_1,\ldots,x_d]$ is Jacobson by \cite[Theorem 3]{goldman}, Theorem \ref{thm2} implies that
the ring $\HH[x_1,\ldots,x_d]=\HH_{\RR[x_1,\ldots,x_d]}$ is  ``left Jacobson\rq\rq. 
%This proves claim (3) of Theorem \ref{thmb} and (together with  Theorem \ref{thm1}) implies claim (4).

We will split the proof of Theorem \ref{thm2} into several lemmas.

\begin{lem}
\label{lem4}
Let $I$ be a left ideal of $\hr$ and let $\mathfrak{m}$ be a maximal ideal of $R$ such that $I \cap R \subseteq \mathfrak{m}$.
There exists a maximal left ideal $M$ of $\hr$ such that $I \subseteq M$ and $M \cap R=\mathfrak{m}$.
\end{lem}

\begin{proof}
If we can prove that the left ideal $L:=I+\hr \mathfrak{m}$ is proper then there exists a maximal left ideal $M$ of $\hr$
which contains $L$. By definition, $M$ also contains $I$ and $\mathfrak{m}$. Since $\mathfrak{m} \subseteq M \cap R$
and $\mathfrak{m}$ is maximal, it follows that $M \cap R=\mathfrak{m}$ as required.

Suppose that $L$ is not proper. Then there exist $c \in I$ and $z \in \hr \mathfrak{m}$ such that $1=c+z$.
Write $c=c_0+c_1 \ii+c_2 \jj+c_3 \kk$ and $z=z_0+z_1 \ii+z_2 \jj+z_3 \kk$. By comparing coefficients, we obtain
\begin{equation}
\label{eq1}
1=c_0+z_0, \quad 0=c_1+z_1, \quad 0=c_2+z_2, \quad 0=c_3+z_3.
\end{equation}
Since $c \in I$ and $I$ is a left ideal we have $\bar{c}c \in I$ implying
\begin{equation}
\label{eq2}
c_0^2+c_1^2+c_2^2+c_3^2 \in I \cap R \subseteq \mathfrak{m}.
\end{equation}
Since $z \in \hr \mathfrak{m}$  and $\mathfrak{m}$ is an ideal of $R$ we have 
\begin{equation}
\label{eq3}
z_0,z_1,z_2,z_3 \in \mathfrak{m}.
\end{equation}
By \eqref{eq1} and \eqref{eq3} we have $c_1,c_2,c_3 \in \mathfrak{m}$ and $1-c_0 \in \mathfrak{m}$. 
Now \eqref{eq2} implies that $c_0^2 \in \mathfrak{m}$ which implies that $c_0 \in \mathfrak{m}$.
It follows that $1=c_0+(1-c_0) \in \mathfrak{m}$, a contradiction. Therefore $L$ is proper.
\end{proof}

As a trivial application of Lemma \ref{lem4} and the Jacobson property of $R$ we obtain Corollary \ref{cor1}.

\begin{cor}
\label{cor1}
Let $I$ be a prime left ideal of $\hr$ and let $J$ be the intersection of all maximal left ideals of $\hr$ that contain $I$.
Then ${I \cap R}={J \cap R}$.
\end{cor}

%To prove Theorem \ref{thm2} we need to show that $I  \cap R=J \cap R$ and $I \subseteq J$ imply that $I=J$ if $I$ is prime.
%If $I$ is a two-sided ideal, this follows from Lemma \ref{lem6}
%and if $I$ is not two-sided this follows from Lemma \ref{lem5}.

\begin{rem}
\label{rem0}
In the proofs of Lemmas \ref{lem5} and \ref{lem6} we will use the 1-1 correspondence 
between the ideals of $R$ and the two-sided ideals of $\HH_R$
given by $\mathfrak{a} \mapsto \hr \mathfrak{a}$ and $I \mapsto I \cap R$.
This correspondence follows from Remark \ref{rem00}; see Lemma 4.2 in \cite{ap1}.
Note also that it sends prime ideals to prime ideals. Namely, if $I \cap R$ is prime
and  $a,b \in  \hr$ satisfy $a \hr b \subseteq I$ then $a_i b_j \in I \cap R$ for all $i,j$ by Remark \ref{rem00}.
It follows that either $a_i \in I \cap R$ for all $i$ or $b_j \in I \cap R$ for all $j$
which implies that either $a \in I$ or $b \in I$. The other direction is clear.
\end{rem}

Lemma \ref{lem5} is a minor variation of \cite[Lemma 4.3]{ap1}. 

\begin{lem}
\label{lem5}
Let $I$ be a prime left ideal of $\hr$ and let $J$ be a left ideal of $\hr$ such that $I \subseteq J$ and $I \cap R=J \cap R$.
If $I$ is not a two-sided ideal then $I=J$.
\end{lem}

\begin{proof}
For every $a \in J$ we have $\bar{a}a \in J \cap R  \subseteq I$. It follows that for every $b \in J$ and $c \in I$,
$\bar{c}b=(\overline{b+c})(b+c)-\bar{b}b-\bar{c}c-\bar{b}c \in I$, i.e.
\begin{equation}
\label{eq4}
\bar{I}J \subseteq I.
\end{equation}
%By conjugation we get $\bar{J}I \subseteq \bar{I}$. Since $\bar{I}$ is a right ideal, it follows that $\bar{J}I \hr \subseteq \bar{I}$. 
The ideal $\mathfrak{p}':=I \hr \cap R$ is different from the ideal $\mathfrak{p}:=I \cap R$.
(By Remark \ref{rem0}, $I \hr=\hr \mathfrak{p}'$. If $\mathfrak{p}'=\mathfrak{p}$ then $I \hr =\hr \mathfrak{p}\subseteq I$ 
which contradicts the assumption that $I$ is not two-sided.) 
Pick $a \in \mathfrak{p}' \setminus \mathfrak{p}$ and note that $a =\bar{a} \in \hr \bar{I}$.
To prove that $J \subseteq I$ pick $x \in J$. By \eqref{eq4} we have  $a x \in \hr \bar{I}J \subseteq I$. Since $I$ is prime and $a \in R \setminus I$,
it follows by Lemma \ref{lem1} that $x \in I$. Therefore $I=J$.
\end{proof}

\begin{lem}
\label{lem6}
Every two-sided prime ideal in $\hr$ is an intersection of maximal left ideals.
\end{lem}

\begin{proof}
Since $R$ is a Jacobson ring, it follows by Remark \ref{rem0} that every two-sided prime ideal of $\hr$ is equal to 
an intersection of maximal two-sided ideals. 

It remains to show that every maximal two-sided ideal $M$ of $\hr$ is an intersection of maximal left ideals.
Namely, since $\hr/M$ is a simple ring, it has a trivial Jacobson radical. Therefore, the intersection of maximal
left ideals in $\hr/M$ is equal to zero which implies the claim.
\end{proof}

We are now ready for the proof of Theorem \ref{thm2}.

\begin{proof}
Let $I$ be a prime left ideal of $\hr$ and let $J$ be the intersection of all maximal left ideals of $\hr$ that contain $I$.
By Corollary \ref{cor1},  $I \cap R=J \cap R$, and clearly $I \subseteq J$. 
%We want to show that $I=J$.
%To prove Theorem \ref{thm2} we need to show that $I  \cap R=J \cap R$ and $I \subseteq J$ imply that $I=J$ if $I$ is prime.
If $I$ is a two-sided ideal then $I=J$ by Lemma \ref{lem6}.
If $I$ is not two-sided then $I=J$ by Lemma \ref{lem5}.
\end{proof}

Theorem  \ref{thm3} is our first main result.
It rephrases claim (4) of Theorem \ref{thmb}.
 It is a corollary of Theorems \ref{thm1} and \ref{thm2}.
An alternative proof will be given by Theorem \ref{appthm}.

\begin{thm}
\label{thm3}
For every $f, g_1,\ldots,g_m \in \HH[x_1,\ldots,x_d]$, the following statements are equivalent:
\begin{enumerate}
\item For every $a \in \HH_c^d$, if $g_1(a)=\ldots=g_m(a)=0$ then $f(a)=0$.
\item $f$ belongs to the smallest semiprime left ideal containing  $g_1,\ldots,g_m$.
\end{enumerate}
\end{thm}

\begin{proof}
Write $I$ for the left ideal generated by $g_1,\ldots,g_m$. Clearly, claim (1) is equivalent to $f \in \cap_{I \subseteq I_a} I_a$.
By Theorems \ref{thma}, \ref{thm1} and \ref{thm2}, $\cap_{I \subseteq I_a} I_a$ is equal to the smallest semiprime left ideal
that contains $I$.
\end{proof}

\section{Completely prime left ideals}
\label{sec3}

The aim of this section is to extend  \cite[Theorem 1.2]{ap1} from quaternionic to matrix polynomials;
i.e. to prove claim (2) of Theorem \ref{thmd}. See Corollary \ref{cor3}.

Let us recall from \cite{reyes} the definition of a completely prime left ideal.
A left ideal $I$ of an associative unital ring $A$ is \textit{completely prime} 
if for every $a,b \in A$ such that $ab \in I$ and $Ib \subseteq I$ we have $a \in A$ or $b \in A$. 

In Lemmas \ref{lem7a} and \ref{lem7b} we observe the following inclusions
$$
\left\{ {\mbox{maximal left} \atop \mbox{ideals of } M_n(R)}  \right\} \subseteq 
\left\{ {\mbox{completely prime} \atop \mbox{left ideals of } M_n(R)}  \right\} \subseteq 
\left\{ {\mbox{prime left} \atop \mbox{ideals of } M_n(R)}  \right\}
$$
which together with \cite[Theorem 3]{cimpric} imply  Corollary \ref{cor3}.

\begin{lem}
\label{lem7a}
Every maximal left ideal of every associative unital ring is completely prime.
\end{lem}

\begin{proof}
See  \cite[Corollary 2.10]{reyes}.
\end{proof}

Lemma \ref{matlem} is the analogue of Lemmas \ref{lem1}  and \ref{lem2} for  $M_n(R)$.
We will write $\id$ for the identity matrix.

\begin{lem}
\label{matlem}
Let $R$ be a commutative unital ring, $n$ a positive integer and $I$ a left ideal of $M_n(R)$.
\begin{enumerate}
\item[(a)] The left ideal $I$ is prime  iff for every $a \in R$ and $B \in M_n(R)$ such that $aB \in I$ we have either $a \, \id \in I$ or $B \in I$. 
\item[(b)] The left ideal $I$ is semiprime iff for every $A=[a_{i,j}] \in M_n(R)$ such that $a_{i,j}A \in I$ for all $i,j$ we have $A \in I$.
 \end{enumerate}
 \end{lem}
 
 \begin{proof} 
 Suppose that $I$ satisfies the property in claim (a). To prove that $I$ is prime, 
 pick any $A,B\in M_n(R)$ such that $A M_n(R) B \subseteq I$. If $A=[a_{r,s}]$
then $a_{i,j}\id =\sum_{k=1}^n E_{k,i} A E_{j,k}$ for each $i,j$ 
(where $E_{r,s}$ are coordinate matrices). It follows that
$a_{i,j}B=\sum_{k=1}^n E_{k,i} A E_{j,k} B \in I$  for each $i,j$.
If $B \not\in I$, it follows by assumption that $a_{i,j}  \id \in I$ for all $i,j$
which implies that $A=\sum_{i,j} E_{i,j}(a_{i,j} \id) \in I$.
Conversely, if $I$ is prime and $aB \in I$ for some $a \in R$ and $B \in M_n(R)$
then $(a \, \id) M_n(R) B \subseteq I$ implying that $a \, \id \in I$ or $B \in I$.
The proof of (b) is similar.
 \end{proof}

\begin{lem}
\label{lem7b} 
Let $R$ be a commutative unital ring and $n$ a positive integer.
Every completely prime left ideal of $M_n(R)$ is prime.
\end{lem}

\begin{proof}
Suppose that $I$ is a completely prime left ideal of $M_n(R)$.
We will  prove that $I$ is prime by verifying the propery in claim (a) of Lemma \ref{matlem}.
Pick $a \in R$ and $B \in M_n(R)$ such that $a B \in I$. It follows that
$B (a\,\id) \in I$ and $I (a\,\id) \subseteq I$. By the definition of a completely prime ideal
either $a\,\id \in I$ or $B \in I$.  
\end{proof}

We are now ready for the main result of this section.

\begin{cor}
\label{cor3}
Let $R$ be a Jacobson ring and $n$ a positive integer. For every left ideal $I$ of  $M_n(R)$
the following are equivalent:
\begin{enumerate}
\item[(a)] $I$ is an intersection of maximal left ideals,
\item[(b)] $I$ is an intersection of completely prime left ideals,
\item[(c)] $I$ is an intersection of prime left ideals,
\item[(d)] $I$ is semiprime.
\end{enumerate}
\end{cor}

\begin{proof}
By Lemma \ref{lem7a}, (a) implies (b).
By Lemma \ref{lem7b}, (b) implies (c).
Since every prime ideal is semiprime and
every intersection of semiprime ideals is also semiprime, (c) implies (d).
By \cite[Theorem 3]{cimpric}, (d) implies (a). 
\end{proof}

Lemma \ref{lem9} is an analogue of Lemma \ref{lem7b} for $\hr$.

\begin{lem}
\label{lem9}
Every completely prime left ideal of $\hr$ is prime. 
Every prime left ideal of $\hr$ which is not two-sided is completely prime.
However, a prime two-sided ideal of $\hr$  need not be completely prime. 
\end{lem}

\begin{proof} Suppose that $I$ is a completely prime left ideal and $r \in R$ and $a \in \hr$ are such that $ra \in I$.
Clearly, $ar \in I$ and $Ir \subseteq I$ which implies that $r \in I$ or $a \in I$. By Lemma \ref{lem1}, $I$ is prime.

Suppose now that $I$ is a prime left ideal of $\hr$ which is not two-sided. 
To prove that $I$ is completely prime, pick $a,b \in \hr$ such that $ab \in I$ and $I b \subseteq I$.
Write $J=I+\hr a$, $\mathfrak{p}'=J \cap R$ and $\mathfrak{p} =I \cap R$.
If $\mathfrak{p}'=\mathfrak{p}$ then, by Lemma \ref{lem5}, $I=J$  which implies that $a \in I$. 
Otherwise  pick $c \in \mathfrak{p}' \setminus \mathfrak{p}$ and note that $cb \in Jb \subseteq I$.
Since $I$ is prime and $c \in R \setminus I$, it follows by Lemma \ref{lem1} that $b \in I$.

If $R=\RR[x]$ then $\hr  (x^2+1)$ is prime
by Remark \ref{rem0} but it is not completely prime by \cite[Lemma 4.7]{ap1}. 
\end{proof}

We conclude this section by observing that our Theorems \ref{thm1} and \ref{thm2} 
imply the following extension of \cite[Theorem 1.2]{ap1}.

\begin{cor}
\label{cor2} 
Let $R$ be a Jacobson ring containing $\frac12$.
For every left ideal $I$ of  $\hr$ the following are equivalent:
\begin{enumerate}
\item[(a)] $I$ is an intersection of maximal left ideals,
\item[(b)] $I$ is an intersection of completely prime left ideals,
\item[(c)] $I$ is an intersection of prime left ideals,
\item[(d)] $I$ is semiprime.
\end{enumerate}
\end{cor}

\begin{proof}
By Lemma \ref{lem7a}, (a) implies (b). By Lemma \ref{lem9}, (b) implies (c).
Clearly, (c) implies (d). By Theorems \ref{thm1} and \ref{thm2}, (d) implies (a).
\end{proof}

\section{An explicit Nullstellensatz for matrix polynomials}
\label{sec4}

Theorem \ref{thm4} is the second main result of this paper. 
It rephrases claim (1) of Theorem \ref{thmd}.
%It is an extension of  \cite[Theorem 1.5]{aya} from quaternionic to matrix polynomials;
The proof is based on Theorem \ref{thmc} (i.e. the weak nullstellensatz for matrix polynomials)
and the Rabinowitsch trick.
We can deduce  \cite[Theorem 1.5]{aya} from Theorem \ref{thm4}
by embedding $\HH[x_1,\ldots,x_d]$ into $M_2(\CC[x_1,\ldots,x_d])$;
see  Theorem \ref{appthm}.

\begin{thm}
\label{thm4}
Let $\acf$ be an algebraically closed field.
Pick any matrix polynomials $F,G_1,\ldots,G_m \in M_n(\acf[x_1,\ldots,x_d])$ 
and write $I$ for the left ideal  generated by $G_1,\ldots,G_m$. The following are equivalent:
\begin{enumerate}
\item For every point $a \in \acf^d$ and every vector $v \in \acf^n$ such that $G_1(a)v=\ldots=G_m(a)v=0$
we have  $F(a)v=0$.
\item For every constant matrix $A \in M_n(\acf)$  there exists  $N \in  \NN_0$ such that
$(AF)^N \in I+I(AF)+\ldots+I(AF)^N$.
\end{enumerate}
\end{thm}

\begin{proof}
Write $R=\acf[x_1,\ldots,x_d]$ and $R'=R[y]$. Let $I'$ be the left ideal in $M_n(R')$ generated by $I$
and $y F- \id$  where $\id$ is the identity matrix.

Firstly, we show that claim (1) of Theorem \ref{thm4} implies that $\id \in I'$.
Suppose that $\id \not\in I'$. 
%Pick a maximal left ideal $M$ of $M_n(R')$ containing $I'$.
By Theorem \ref{thmc}, there exist a point 
$(a_1,\ldots,a_d,b) \in \acf^{d+1}$ and a nonzero vector $v \in \acf^n$. such that 
$H(a_1,\ldots,a_d,b)v=0$ for every matrix polynomial $H \in I'$. In particular,
$G_i(a_1,\ldots,a_d)v=0$ for all $i$ and $b F(a_1,\ldots,a_n)v-v=0$.
By claim (1), we also have $F(a_1,\ldots,a_n)v=0$ which gives a contradiction $-v=0$. 

Secondly, we show that $\id \in I'$ implies claim (2) of Theorem \ref{thm4}.
Pick  $H_1,\ldots,H_m \in M_n(R')$ and $K \in M_n(R')$ such that
\begin{equation}
\label{i1}
H_1 G_1+\ldots+H_m G_m +K(y F-\id)=\id.
\end{equation}
Write $H_1,\ldots,H_m$ and $K$ as polynomials in $y$ with coefficients in $R$; i.e. 
\begin{equation}
\label{i2}
H_1=\sum_{i=0}^N H_{1,i} y^i, \ldots, H_m=\sum_{i=0}^N H_{m,i} y^i 
\mbox{ and } K=\sum_{i=0}^{N-1} K_i y^i
\end{equation}
where $N \in \NN$ and $H_{k,i}, K_i \in R$ for all $i,k$.
By inserting \eqref{i2} into \eqref{i1} and comparing coefficients at the powers of $y$ we obtain equations
\begin{equation}
\begin{array}{ccc}
 \label{i3} 
H_{1,0}G_1+\ldots+H_{m,0}G_m+(-K_0) & = & \id, \\
H_{1,1}G_1+\ldots+H_{m,1}G_m+(K_0 F-K_1) & = & 0, \\
H_{1,2}G_1+\ldots+H_{m,2}G_m+(K_1 F-K_2) & = & 0, \\
 \vdots &  &  \\
H_{1,N-1}G_1+\ldots+H_{m,N-1}G_m+(K_{N-2}F-K_{N-1}) & = & 0, \\
H_{1,N}G_1+\ldots+H_{m,N}G_m+(K_{N-1} F) & = & 0. \\
\end{array}
\end{equation}
By telescoping \eqref{i3} we obtain 
\begin{equation}
\label{i4}
\sum_{k=0}^n (H_{1,k}G_1+\ldots+H_{m,k}G_m)F^{N-k}=F^N.
\end{equation}
Since $G_1,\ldots,G_m \in I$, equation \eqref{i4} implies that
\begin{equation}
\label{i5}
F^N \in I+IF+\ldots + IF^N.
\end{equation}
If we repeat the proof of \eqref{i5} with $F$ replaced by $AF$ for any $A \in M_n(\acf)$ 
we see that claim (1)  of Theorem \ref{thm4} implies claim (2).

Finally, we prove that claim (2) of Theorem \ref{thm4} implies claim (1).
Assume, for sake of contradiction, that claim (1) is false and 
pick any $a \in \acf^d$ and  $v \in \acf^n$ such that 
$G_1(a)v=\ldots=G_m(a)v=0$ and $F(a)v \ne 0$. 
%\end{equation}
Choose $A \in M_n(\acf)$ such that $A F(a)v=v$. 
By claim (2) there exist $N \in \NN$ and $L_0,L_1,\ldots,L_N \in I$ such that
\begin{equation}
\label{i7}
(AF)^N=L_0+L_1 (AF)+\ldots+L_N (AF)^N.
\end{equation}
If we evaluate \eqref{i7} in $a$ and multiply it by $v$ from the right
then,  by the choice of $A$, we get
\begin{equation}
\label{i8}
v=L_0(a)v+L_1(a) v+\ldots+L_N(a) v. 
% (A F(a))^N v=L_0(a)v+L_1(a) (AF(a))v+\ldots+L_N(a) (AF(a))^N v \\
\end{equation}
By the choice of $a$ and $v$, the left-hand side of \eqref{i8} is non-zero while the right-hand side is zero.
\end{proof}

\appendix
\def\appendixname{\!\!}
\section{Appendix: Another proof of Theorem \ref{thmb}}

We will give an alternative proof of Theorem \ref{thm3} by using \cite[Theorem 4]{cimpric}. 
We will also show that  Theorem \ref{thm4}  implies \cite[Theorem 1.5]{aya}.
In other words, we will show that Theorem \ref{thmb} follows from Theorem \ref{thmd}.

Let $R$ be a commutative unital ring. Its \textit{complexification} is the ring
\begin{equation}
\ccr := \{r_0+r_1 \ii \mid r_0,r_1 \in R\}
\end{equation}
where $\ii^2=-1$. Clearly, $\CC_{\RR[x_1,\ldots,x_d]}=\CC[x_1,\ldots,x_d]$.
Write $\overline{r_0+ r_1 \ii}=r_0- r_1 \ii$.
Every element of $\hr$ can be written as $p+\jj q$ where $p,q \in \ccr$.
Note that we have a homomorphism
\begin{equation}
\label{fm1}
\begin{array}{ccccc}
\varphi &  \colon & \hr & \to & M_2(\ccr), \\
& &  p+ \jj q & \mapsto & \left[ \begin{array}{cc} p & -\bar{q} \\ q &\bar{p} \end{array} \right].
\end{array}
\end{equation}

Lemma \ref{applem1} will be used in Proposition \ref{appprop2} and Theorem \ref{appthm}.

\begin{lem}
\label{applem1}
Let $R$ be a commutative ring containing $\frac12$.
Every element of $M_2(\ccr)$ can be uniquely expressed as
 $\varphi(z)+\ii \varphi(w)$ where $z,w \in \hr$. 
\end{lem}

\begin{proof}
Note that for any $A=\left[\begin{array}{cc} a & b \\ c & d \end{array} \right] \in M_2(\ccr)$ we have
\begin{equation}
\label{izp1} 
A=\frac12\left((a+d)\varphi(1)+\ii(d-a)\varphi(\ii)+(c-b)\varphi(\jj)+\ii(b+c)\varphi(\kk)\right).
\end{equation}
It follows that $\varphi(1),\varphi(\ii),\varphi(\jj),\varphi(\kk),\ii \varphi(1),\ii \varphi(\ii),\ii \varphi(\jj),\ii \varphi(\kk)$
is a basis if we consider $M_2(\ccr)$ as an $R$-module. This implies the claim
as the mapping $\varphi$ is an $R$-module homomorphism.
\end{proof}

Proposition \ref{appprop2} will help us prove that claim (3) of Theorem \ref{appthm} implies claim (4).

\begin{prop}
\label{appprop2}
Let $R$ be a commutative ring containing $\frac12$ and let $J$ be a left ideal of $\hr$.
The set $$J':=\varphi(J)+\ii\varphi(J)$$
is the smallest left ideal of $M_2(\ccr)$ that contains the  set $\varphi(J)$ and $\varphi^{-1}(J')=J$.
If $J$ is semiprime then $J'$ is also semiprime.
\end{prop}

\begin{proof}
Pick a left ideal $J$ of $\hr$ and write $J':=\varphi(J)+\ii\varphi(J)$.
By the existence part of Lemma \ref{applem1} and the fact that $\varphi$ is a homomorphism,
it follows that $J'$ is a left ideal of $M_2(\ccr)$.
By the uniqueness part of Lemma \ref{applem1} we have $\varphi^{-1}(J')=J$.
Every left ideal of $M_2(\ccr)$ that contains $\varphi(J)$
must also contain  $\ii \varphi(J)$ and $J'$.

Suppose now that $J$ is semiprime. To show that $J'$ is also semiprime 
pick $A \in M_2(\ccr)$ such that  for every $X \in M_2(\ccr)$ we have
\begin{equation}
\label{izp4}
AXA \in J'.
\end{equation}
By the existence part of Lemma \ref{applem1}, we can write $A=\varphi(u)+\ii \varphi(v)$ for some 
$u,v \in  \hr$. Therefore, for every $x \in  \hr$,
\begin{equation}
(\varphi(u)+\ii \varphi(v)) \varphi(x)(\varphi(u)+\ii \varphi(v)) \in \varphi(J)+\ii\varphi(J)
\end{equation}
which, by the uniqueness part of Lemma \ref{applem1}, implies that
\begin{equation}
\label{izp4a}
uxu-vxv \in J \quad \mbox{ and } \quad vxu+uxv \in J
\end{equation}
for every $x \in  \hr$. By Remark \ref{rem00} and  \eqref{izp4a} we have
\begin{equation}
\label{izp4b}
 u_k u -v_k v \in J \quad \mbox{ and } \quad v_k u+u_k v  \in J
\end{equation}
for every $k$. It follows that for all $k$ and $l$ 
\begin{equation}
\label{izp4c}
\begin{array}{ccc}
(u_k u_l +v_k v_l)u & = & u_k( u_l u -v_l v )+v_l(v_k u+u_k v) \in J, \\
(u_k u_l +v_k v_l)v & = & -v_k( u_l u -v_l v )+u_l(v_k u+u_k v) \in J.
\end{array}
\end{equation}
From \eqref{izp4c} it follows that for every $z \in \hr$ and every $k$
\begin{equation}
\label{izp4d}
\begin{array}{ccl}
\quad (u_k u+v_k v)zu =  z (u_k u_0+v_k v_0)u+\ii z (u_k u_1+v_k v_1)u+ \\
 +\jj z (u_k u_2+v_k v_2)u+\kk z (u_k u_3+v_k v_3)u \in J
\end{array}
\end{equation}
and by the same argument
\begin{equation}
\label{izp4e}
(u_k u+v_k v)zv \in J.
\end{equation}
Equations \eqref{izp4d} and \eqref{izp4e} imply that for each $k$ and $z$
\begin{equation}
\label{izp4f}
(u_k u+v_k v)z(u_k u+v_k v)\in J.
\end{equation}
Since $J$ is semiprime, it follows that for each $k$
\begin{equation}
\label{izp4g}
u_k u+v_k v \in J.
\end{equation}
Equation \eqref{izp4b} and \eqref{izp4g} imply that for each $k$
\begin{equation}
\label{izp4h}
u_k u \in J \quad \mbox{ and } \quad v_k v \in J.
\end{equation}
Since $J$ is semiprime, equations \eqref{izp4h} imply that $u \in J$ and $v \in J$ by Lemma \ref{lem2}.
It follows that $A=\varphi(u)+\ii \varphi(v) \in J'$ as claimed.
\end{proof}

Theorem \ref{appthm} explains how Theorem \ref{thmb} can be deduced from Theorem \ref{thmd}.
The plan of the proof is to show (1) $\Rightarrow$ (2)  $\Rightarrow$ (3)  $\Rightarrow$ (4)  $\Rightarrow$ (1)
and  (1) $\Rightarrow$ (2)  $\Rightarrow$ (5)  $\Rightarrow$ (6)  $\Rightarrow$ (1).
Implications  (2)  $\Rightarrow$ (3) and  (2)  $\Rightarrow$ (5) are provided by Theorem \ref{thmd}.
%while  (3)  $\Rightarrow$ (4) uses  Proposition \ref{appprop2}.

\begin{thm}
\label{appthm}
Let $I$ be a left ideal of $\HH[x_1,\ldots,x_d]$ and let  $I'=\varphi(I)+\ii\varphi(I)$.
For every $f \in \HH[x_1,\ldots,x_d]$ the following are equivalent.
\begin{enumerate}
\item $f \in \sqrt{I}$,
\item $\varphi(f) \in \sqrt{I'}$,
\item $\varphi(f)$ belongs to the smallest semiprime left ideal containing $I'$.
\item $f$ belongs to the smallest semiprime left ideal containing $I$.
\item For every matrix $A \in M_2(\CC)$  there exists  $N \in  \NN_0$ such that
$$(A\varphi(f))^N \in I'+I'(A\varphi(f))+\ldots+I'(A\varphi(f))^N.$$
\item For every $b \in \HH$ there exists $N \in \NN_0$ such that 
$$(bf)^N \in I+I(bf)+\ldots+I(bf)^N.$$
\end{enumerate}
\end{thm}

\begin{proof} 
Assume that claim (1) is true. To prove claim (2) pick any 
 $c \in \CC^d$ and  $0  \ne w \in \CC^2$ such that $I' \subseteq D(c,w)$. 
We must show that $\varphi(f) \in D(c,w)$.

For every $g \in I$ we have $\varphi(g)(c)w=0$.
Write $g=h+\jj k$ where $k,h \in \CC[x_1,\ldots,x_d]$
and $w=[u \ v]^T$ where $u,v \in \CC$ are not both zero.
By  \eqref{fm1} we get
\begin{equation}
\label{fm3}
h(c)u-\bar{k}(c)v=0 \quad \mbox{ and } \quad k(c)u+\bar{h}(c)v=0.
\end{equation}
Write $b=u+\jj v$ and $a =b c b^{-1}$ and note that $a \in \HH_c^d$.
By \eqref{fm3},
\begin{equation}
\label{fm4}
g(a)b  =  (g b)(c) =((h+\jj k)(u+\jj v))(c)=\big(h u-\bar{k} v+\jj (k u+\bar{h} v)\big)(c) =0
\end{equation}
As $b \ne 0$ we have $g(a)=0$. This proves that $I \subseteq I_a$. Therefore, $f \in I_a$.
If we reverse the computation above, we see that  $f(a)b=0$ implies $\varphi(f)(c)w=0$.

By  \cite[Theorem 4]{cimpric},  claim (2) implies claim (3).

To prove that claim (3) implies claim (4) pick $f \in \HH[x_1,\ldots,x_d]$ such that $\varphi(f)$ 
belongs to the smallest semiprime left ideal of $M_2(\CC[x_1,\ldots,x_d])$ that contains $I'$.  
Write $J$ for  the smallest semiprime left ideal of $\HH[x_1,\ldots,x_d]$ that contains $I$. 
The left ideal $J':=\varphi(J)+\ii \varphi(J)$ is semiprime by the second part of Proposition \ref{appprop2}
and it clearly contains $I'$. Therefore $\varphi(f)$ belongs to $J'$ by the choice of $f$.
It follows that $f$ belongs to  $J$ by the first part of Proposition \ref{appprop2}.

Since every evaluation ideal $I_a$ defined by formula \eqref{def3} is semiprime,
claim (4) implies claim (1). 

By Theorem \ref{thm4}, claim (2) implies claim (5).

Assume that claim (5) is true. To prove claim (6) pick any $b \in \HH$ and apply claim (5) with $A=\varphi(b)$.
Since $I'=\varphi(I)+\ii \varphi(I)$ by the first part of Proposition \ref{appprop2}, 
there exist $N \in \NN$ and $u_k,v_k \in I$ for $k=0,1,\ldots,N$ such that
\begin{equation}
(\varphi(b) \varphi(f))^N=\sum_{k=0}^N (\varphi(u_k)+\ii \varphi(v_k)) (\varphi(b) \varphi(f))^k.
\end{equation}
By the uniqueness part of Lemma \ref{applem1},
\begin{equation}
\label{auxapp}
(\varphi(b) \varphi(f))^N=\sum_{k=0}^N \varphi(u_k) (\varphi(b) \varphi(f))^k,
\end{equation}
which implies claim (6) since $\varphi$ is one-to-one.

%If claim (1) is false then there is $a \in \HH_c^d$ such that
%$g_1(a)=\ldots=g_m(a)=0$ and $f(a) \ne 0$. 
%It follows that claim (6) is false for $b=f(a)^{-1}$ because by 
%evaluationg both sides at $a$ we get $1=0$ by formula \eqref{eval}.

It was already proved in \cite[Theorem 1.5]{aya} that claim (6) implies claim (1).
\end{proof}

\section*{Acknowledgement}
We would like to thank the anonymous referees for their thorough review and insightful comments. Their feedback was instrumental in identifying and correcting several errors, suggesting improvements to the presentation, and providing interesting conjectures.

\end{document}